\documentclass[11]{amsart}
\usepackage{amscd,amssymb}
\usepackage[all]{xy}
\usepackage{graphicx,verbatim}
\usepackage[colorlinks,plainpages,backref,urlcolor=blue]{hyperref}

\newtheorem{theorem}{Theorem}[section]
\newtheorem{corollary}[theorem]{Corollary}
\newtheorem{lemma}[theorem]{Lemma}
\newtheorem{prop}[theorem]{Proposition}

\theoremstyle{definition}
\newtheorem{definition}[theorem]{Definition}
\newtheorem{example}[theorem]{Example}
\newtheorem{remark}[theorem]{Remark}
\newtheorem*{ack}{Acknowledgments}

\newcommand{\Z}{\mathbb{Z}}
\newcommand{\Q}{\mathbb{Q}}

\newcommand{\C}{\mathbb{C}}

\newcommand{\PP}{\mathbb{P}}
\newcommand{\GG}{\mathbb{G}}
\newcommand{\K}{\mathbb{K}}

\newcommand{\RR}{{\mathcal R}}

\newcommand{\Wey}{\mathcal{W}}

\newcommand{\B}{{\mathfrak{B}}}
\newcommand{\g}{{\mathfrak{g}}}
\newcommand{\h}{{\mathfrak{h}}}
\newcommand{\gl}{{\mathfrak{gl}}}
\renewcommand{\sl}{{\mathfrak{sl}}}
\newcommand{\m}{{\mathfrak{m}}}

\renewcommand{\sp}{{\mathfrak{sp}}}

\newcommand{\BB}{\widetilde{B}}

\newcommand{\V}{\mathsf{V}}
\newcommand{\W}{\mathsf{W}}

\newcommand{\G}{{\Gamma}}
\newcommand{\T}{{\mathcal{T}}}

\DeclareMathOperator{\gr}{gr}
\DeclareMathOperator{\im}{im}

\DeclareMathOperator{\codim}{codim}
\DeclareMathOperator{\id}{id}
\DeclareMathOperator{\ab}{{ab}}
\DeclareMathOperator{\abf}{{abf}}
\DeclareMathOperator{\Sym}{Sym}

\DeclareMathOperator{\GL}{GL}
\DeclareMathOperator{\SL}{SL}
\DeclareMathOperator{\Sp}{Sp}
\DeclareMathOperator{\Hom}{{Hom}}
\DeclareMathOperator{\Tors}{{Tors}}

\DeclareMathOperator{\Out}{Out}
\DeclareMathOperator{\Inn}{Inn}
\DeclareMathOperator{\Aut}{Aut}
\DeclareMathOperator{\OA}{OA}

\DeclareMathOperator{\df}{def}

\DeclareMathOperator{\height}{ht}

\DeclareMathOperator{\Lie}{Lie}
\DeclareMathOperator{\supp}{supp}

\DeclareMathOperator{\Grass}{Gr}
\DeclareMathOperator{\Tor}{Tor}

\DeclareMathOperator{\VV}{\WW^{+}}
\DeclareMathOperator{\WW}{{Weight}}

\newcommand{\same}{\Longleftrightarrow}
\newcommand{\surj}{\twoheadrightarrow}
\newcommand{\inj}{\hookrightarrow}

\newcommand{\isom}{\xrightarrow{\,\simeq\,}}
\newcommand{\abs}[1]{\left| #1 \right|}

\begin{document}
\date{March 23, 2013}

\title[Vanishing resonance and representations of Lie algebras]{%
Vanishing resonance and representations of \\ Lie algebras}

\author[Stefan Papadima]{Stefan Papadima$^1$}
\address{Simion Stoilow Institute of Mathematics, 
P.O. Box 1-764,
RO-014700 Bucharest, Romania}
\email{Stefan.Papadima@imar.ro}
\thanks{$^1$Partially supported by 
PN-II-ID-PCE-2011-3-0288, grant 132/05.10.2011}

\author[Alexander~I.~Suciu]{Alexander~I.~Suciu$^2$}
\address{Department of Mathematics,
Northeastern University,
Boston, MA 02115, USA}
\email{a.suciu@neu.edu}
\thanks{$^2$Partially supported by 
NSF grant DMS--1010298}



\begin{abstract}
We explore a relationship between the classical representation 
theory of a complex, semisimple Lie algebra $\g$ and the 
resonance varieties $\RR(V,K)\subset V^*$ attached to irreducible 
$\g$-modules $V$ and submodules $K\subset V\wedge V$.
In the process, we give a precise roots-and-weights criterion insuring 
the vanishing of these varieties, or, equivalently, the finiteness 
of certain modules $\Wey(V,K)$ over the symmetric algebra 
on $V$.  In the case when $\g=\sl_2(\C)$, our approach 
sheds new light on the modules studied by Weyman 
and Eisenbud in the context of Green's conjecture 
on free resolutions of canonical curves. In the 
case when $\g=\sl_n(\C)$ or $\sp_{2g}(\C)$, 
our approach yields a unified proof of two 
vanishing results for the resonance varieties 
of the (outer) Torelli groups of surface groups,  
results which arose in recent work  by Dimca, 
Hain, and the authors on homological finiteness 
in the Johnson filtration of mapping class groups 
and automorphism groups of free groups. 
\end{abstract}

\maketitle
\setcounter{tocdepth}{1}
\tableofcontents

\section{Introduction}
\label{sect:intro}

\subsection{Koszul modules and resonance varieties}
\label{subsec:kr}

We start with a simple, yet very general construction.  
Let $V$ be a non-trivial, finite-dimensional complex 
vector space, and let $K\subset V\wedge V$ be a 
subspace. To these data, we associate two objects: 
\begin{itemize}
\item
The {\em Koszul module}, $\Wey(V,K)$, is a graded module over 
the symmetric algebra $S=\Sym(V)$, given by an explicit 
presentation involving the third Koszul differential and the inclusion 
map of $K$ into $V\wedge V$.  
\item
The {\em resonance variety}, $\RR(V,K)$, is a homogeneous 
subvariety inside the dual vector space $V^*$, consisting of 
all elements $a\in V^*$ for which there is an element $b\in V^*$, 
not proportional to $a$, such that $a\wedge b$ belongs to the 
orthogonal complement $K^{\perp}\subseteq V^*\wedge V^*$.
\end{itemize}  
These two objects are closely related: at least away from the 
origin, the resonance variety is the support of the Koszul module.  

We investigate here conditions insuring that the resonance 
variety vanishes, i.e., that the Koszul module is finite-dimensional 
over $\C$.  It is easily seen that $\RR(V,K)=\{0\}$ if and only if 
the plane $\PP(K^{\perp})$ misses the  image under the 
Pl\"ucker embedding in $\PP(V^*\wedge V^*)$  
of the Grassmannian of $2$-planes in $V^*$. 
Thinking of $K$ as a point in the Grassmannian 
of $m$-planes in $V\wedge V$, where  $m=\dim K$, we prove in 
Propositions \ref{prop:unm}, \ref{prop:genvanish}, 
and \ref{prop:gen} the following result. 
 
\begin{theorem}
\label{thm:intro1}
For any integer $m$ with $0\le m\le \binom{n}{2}$, where $n=\dim V$,  
the set 
\begin{equation*}
\label{eq:uset}
U_{n,m}=\big\{ K \in \Grass_m (V\wedge V)
\mid \RR(V,K)=\{0\}  \big\}
\end{equation*}
is Zariski open. Moreover, this set is non-empty if and 
only if $m\ge 2n-3$, in which case there is an integer 
$q=q(n,m)$ such that $\Wey_q(V,K)=0$, 
for every $K\in U_{n,m}$.
\end{theorem}

\subsection{Roots, weights, and vanishing resonance}
\label{subsec:delta}
After these general considerations, we narrow our focus, 
and analyze in detail the case when the vector spaces 
$V$ and $K$ are representations of a complex, semisimple 
Lie algebra $\g$.   To start with, we fix a Cartan subalgebra 
$\h\subset \g$ and a set of simple roots $\Delta\subset \h^*$, 
and we denote by $(\, , \,)$ the inner product on $\h^*$ defined 
by the Killing form.

Let $V=V(\lambda)$ be an irreducible $\g$-module 
corresponding to a dominant weight $\lambda$. 
In Lemmas \ref{lem:weights} and \ref{lem:qqq2}, we 
analyze the irreducible summands in $V(\lambda)\wedge V(\lambda)$  
associated to weights of the form $2\lambda -\beta$, for some 
$\beta\in \Delta$: such summands occur if and only 
if $(\lambda,\beta)\ne 0$.  

Returning  to the main theme, we give in 
Proposition \ref{prop:non-vanish} and Theorem \ref{thm:root vanish} 
two very precise criteria that relate the vanishing of the resonance 
varieties associated to $\g$-modules as above to the roots 
and weights of the Lie algebra $\g$.  To summarize these criteria, 
recall that each simple root $\beta\in \Delta$ gives rise to elements 
$x_\beta, y_\beta \in \g$ and $h_\beta \in \h$ which generate a 
subalgebra of $\g$ isomorphic to $\sl_2(\C)$.

\begin{theorem}
\label{thm:intro3}
Let $V=V(\lambda)$ be an irreducible $\g$-module, and 
let $K\subset V\wedge V$ be a submodule. 
Let $V^*=V(\lambda^*)$ be the dual module, 
and let $v_{\lambda^*}$ be a maximal vector for $V^*$. 
\begin{enumerate}
\item 
Suppose there is a root $\beta\in \Delta$ such that 
$(\lambda^*,\beta)\ne 0$, and suppose the vector 
$v_{\lambda^*} \wedge y_{\beta} v_{\lambda^*}$ 
(of weight $2\lambda^*-\beta$) 
belongs to $K^{\perp}$. Then $\RR(V,K)\ne \{0\}$. 

\item 
Suppose that $2\lambda^* -\beta$ is not a 
dominant weight for $K^{\perp}$, for any simple 
root $\beta$. 
Then $\RR(V,K)=\{0\}$.  
\end{enumerate}
\end{theorem}

As a special case of this theorem, we single out in  
Corollary \ref{cor:mult free} the following situation. 

\begin{corollary}
\label{cor:m1}
In the setup from above, 
$\RR(V,K)=\{0\}$ if and only if 
$2\lambda^* - \beta$ is not a dominant weight 
for $K^{\perp}$, for any simple root $\beta$ 
such that $(\lambda^*,\beta)\ne 0$.
\end{corollary}

In the case when $\g=\sl_2(\C)$, all irreducible representations 
are of the form $V_n=V(n\lambda_1)$, for some $n\ge 0$, where 
$V_1$ is the defining representation.  Moreover, $V_n\wedge V_n$ 
decomposes into multiplicity-free irreps of the form $V_{2n-2-4j}$, 
with $j\ge 0$. The modules $W(n)=\Wey(V_n, V_{2n-2})$ 
have been studied by J.~Weyman and D.~Eisenbud, in an attempt 
to establish an improved version of a conjecture of M.~Green 
regarding the syzygies of canonically embedded curves, see \cite{Ei}.    

In Corollary~\ref{cor:fw}, we recover and strengthen a result 
from \cite{Ei}, as follows: For any $\sl_2(\C)$-submodule 
$K\subset V_n\wedge V_n$, the corresponding Koszul 
module $\Wey(V_n,K)$ is finite-dimensional over $\C$  
if and only if $\Wey(V_n,K)$ is a quotient of $W(n)$. 
The determination of the Hilbert series of these graded 
modules remains an interesting open problem.

\subsection{Alexander invariants, resonance, and Torelli groups}
\label{subsec:aut} 

We switch now to invariants associated to a finitely generated 
group $G$. Set  $V=H_1(G,\C)$, 
identify $V^*=H^1(G,\C)$, and declare $K^{\perp}$ 
to be the kernel of the cup-product map,  
$\cup_G\colon V^*\wedge V^*\to H^2(G,\C)$. 
In this case, the variety $\RR(G)=\RR(V,K)$ coincides 
with the first resonance variety of the group. 
Moreover, the Koszul module, $\B(G)=\Wey(V,K)$,  
can be thought of as an infinitesimal version of the 
classical Alexander invariant, $B(G)=H_1(G',\C)$, 
viewed as a module over the group algebra $\C[G_{\ab}]$, 
or its reduced version, $\BB(G)$, viewed as a module over 
$\C[G_{\ab}/\Tors(G_{\ab})]$. The $I$-adic associated graded
objects, $\gr_{\scriptscriptstyle{\bullet}\!} B(G)$ and 
$\gr_{\scriptscriptstyle{\bullet}\!} \BB(G)$, are both
graded modules over the polynomial algebra $\Sym (V)$.

Let $\Aut(G)$ be the automorphism group of $G$, 
and let $\Out(G)=\Aut(G)/\Inn(G)$ be its outer automorphism 
group.  The Torelli group, $\T_G$, is the kernel of the 
natural homomorphism $\Out(G)\to \Aut(G_{\ab})$
which sends an (outer) automorphism to the induced 
map on the abelianization.  Using our machinery, 
we give a unified proof of two key results from \cite{DP} 
and \cite{PS-johnson} concerning vanishing resonance 
for the Torelli groups of surface groups, and recover
some additional information on their Alexander invariants. 
 
First, we consider the free group on $n\ge 4$ generators;  
its Torelli group, $\OA_n=\T_{F_n}$, is finitely generated.  
Work of Andreadakis, Cohen--Pakianathan, Farb, and Kawazumi 
shows that the natural action of $\GL_n (\Z)$ on $H_1(\OA_n,\Z)$ 
defines an irreducible $\sl_n(\C)$-module structure on $V=H_1(\OA_n,\C)$.  
Furthermore, work of Pettet \cite{Pe} identifies 
$V^*$ with $V( \lambda_1 +\lambda_{n-2})$ and 
$K^{\perp} = \ker(\cup_{\OA_n})$ with 
$V(\lambda_1+\lambda_{n-2}+\lambda_{n-1})$, 
where $\lambda_1,\dots  ,\lambda_{n-1}$ are the 
fundamental dominant weights for $\sl_n(\C)$.
With this notation, we prove in Theorem \ref{thm:rroan} 
and Corollary \ref{cor:boan} the following result.

\begin{theorem}
\label{thm:tfn}
For each $n\ge 4$, the resonance variety $\RR(\OA_n)=\RR(V,K)$ 
vanishes. Moreover, 
$\dim \Wey(V,K) <\infty$ and $\dim \gr_q B(\OA_n) \le \dim \Wey_q(V,K)$, for 
all $q\ge 0$.
\end{theorem}

Next, we consider the fundamental group of a Riemann 
surface of genus $g\ge 3$.  The corresponding Torelli group, 
$\T_g=\T_{\pi_1(\Sigma_g)}$, is known to be finitely generated. 
Work of D.~Johnson \cite{J3} shows that the natural 
$\Sp_{2g} (\Z)$-action on $H_1(\T_g,\Z)$ defines 
an irreducible $\sp_{2g}(\C)$-module structure on $V=H_1(\T_g,\C)$.    
Work of  Hain \cite{Ha} identifies $V^*$ with 
$V(\lambda_3)$ and $K^{\perp}$ with $V(2\lambda_2)\oplus V(0)$, 
where $\lambda_1,\dots  ,\lambda_g$ are the 
fundamental dominant weights for $\sp_{2g}(\C)$.
In Theorem \ref{thm:rrtor} and Corollary \ref{cor:btg}, 
we prove the following result.

\begin{theorem}
\label{thm:tg}
For each $g\ge 4$, the resonance variety $\RR(\T_g)=\RR(V,K)$ vanishes. 
Moreover, if $g\ge 6$,  then 
$\dim \widetilde{B} (\T_g) = \dim \Wey (V,K)<\infty$. 
\end{theorem}

\section{Koszul modules and resonance varieties}
\label{sect:res}

\subsection{Koszul modules}
\label{subsec:wey}
Let $V$  be a non-zero, finite-dimensional vector space 
over $\C$.  Let $S=\Sym(V)$ be the symmetric algebra 
on $V$, let $\bigwedge V$ be the exterior algebra on $V$, 
and let $(S\otimes_{\C} \bigwedge V, \delta)$ be the Koszul resolution, 
where the differential $\delta_p\colon S\otimes_{\C} \bigwedge^p V \to 
S\otimes_{\C} \bigwedge^{p-1} V$ is the $S$-linear map given 
on basis elements by  
\begin{equation}
\label{eq:delta}
\delta_p( v_{i_1}\wedge \cdots \wedge v_{i_p}) = 
\sum_{j=1}^{p} (-1)^{j-1} v_{i_j} \otimes ( v_{i_1}\wedge \cdots 
\wedge \widehat{v_{i_j}} \wedge \dots \wedge v_{i_p}).
\end{equation}

\begin{definition}
\label{def:koszul mod}
The {\em Koszul module}\/ associated to a linear subspace  
$K\subseteq V\wedge V$ is the $S$-module $\Wey(V,K)$ presented as 
\begin{equation} 
\label{eq:kmod}
\xymatrix{
S \otimes_{\C} \big(\bigwedge^3 V \oplus K \big) 
\ar^(.55){\delta _3 + \id \otimes \iota}[rr] &&
S \otimes_{\C}  \bigwedge^2 V \ar@{->>}[r] & \Wey(V,K)
},
\end{equation}
where the group $\bigwedge^3 V$ is in degree $1$, 
the groups $K$ and $\bigwedge^2 V$ are in degree $0$, 
and $\iota \colon K \to V\wedge V$ is the inclusion map.
\end{definition}

Alternatively, we may write
\begin{equation} 
\label{eq:wmod again}
\Wey(V,K) = \im(\delta_2)/\im(\delta_2 \circ (\id \otimes \iota)).
\end{equation}

The $S$-module $\Wey=\Wey(V,K)$ inherits a natural grading 
from the free $S$-module $S \otimes_{\C}  \bigwedge^2 V$; 
we will denote by $\Wey_q$ its $q$-th graded piece.  

\begin{remark}
\label{rem:weyq}
Note that $\Wey$ is generated as an $S$-module by its degree $0$ 
piece, $\Wey_0=(V\wedge V)/K$.  Thus, 
the $S$-module $\Wey(V,K)$ is a finite-dimensional $\C$-vector space 
if and only if there is an integer $q\ge 0$ such that $\Wey_q(V,K)=0$. 
\end{remark}

The above construction enjoys some nice functoriality properties.
For instance, suppose $(V'\wedge V',K')$ is another pair of 
vector spaces as above, and suppose there is a linear map 
$\varphi\colon V\to V'$ such that 
$\varphi \wedge \varphi\colon V\wedge V\to V'\wedge V'$ 
takes $K$ to $K'$.  Write $S'=\Sym(V')$, and let 
$\Sym(\varphi)\colon S\to S'$ be the induced ring morphism.  
The commuting diagram
\begin{equation} 
\label{eq:gen weyman}
\xymatrix{
S \otimes_{\C} \big(\bigwedge^3 V \oplus K \big) 
\ar^{\Sym(\varphi)\otimes (\bigwedge^3\varphi \oplus \bigwedge^2\varphi |_{K})}[d]
\ar^(.55){\delta _3 + \id \otimes \iota}[rr] &&
S \otimes_{\C}  \bigwedge^2 V \ar@{->>}[r] 
 \ar^{\Sym(\varphi)\otimes \bigwedge^2\varphi}[d] 
& \Wey(V,K)\,  \ar@{-->}^{\Wey(\varphi)}[d]\\
S' \otimes_{\C} \big(\bigwedge^3 V' \oplus K' \big) 
\ar^(.55){\delta'_3 + \id \otimes \iota'}[rr] &&
S' \otimes_{\C}  \bigwedge^2 V' \ar@{->>}[r] & \Wey(V',K'). 
}
\end{equation}
defines a $\Sym(\varphi)$-equivariant map $\Wey(\varphi)$ 
between the respective Koszul modules. 
Note that, if $\varphi$ is surjective, then $\Wey(\varphi)$ is also surjective.

\subsection{Resonance varieties }
\label{subsec:res VK}

Let $V^*=\Hom_{\C}(V,\C)$ be the dual vector space, 
and let $K^{\perp}\subseteq V^*\wedge V^* = (V\wedge V)^*$ 
be the linear subspace consisting of all functionals 
that vanish identically on the subspace $K\subseteq V\wedge V$.

\begin{definition}
\label{def:res}
The {\em resonance variety}\/ associated to the pair $(V\wedge V,K)$ 
is defined as 
\begin{equation}
\label{eq:rvk again}
\RR(V,K) =\{ a \in V^* \mid \text{$\exists\, b \in V^*$ 
such that $a\wedge b\ne 0$ and $a\wedge b \in K^{\perp}$}\} \cup \{0\}.
\end{equation}
\end{definition}

It is readily seen that $\RR(V,K)$ is a conical, Zariski-closed 
subset of the affine space $V^*$. 
For instance, if $K=0$ and if $\dim V>1$, then $\RR(V,0) =V^*$. 
At the other extreme, $\RR(V,V\wedge V)=\{0\}$. 

This construction also enjoys some pleasant functoriality properties.
For instance, suppose $\varphi\colon V\surj V'$ is a surjective 
linear map inducing a morphism from $(V\wedge V,K)$ to   
$(V'\wedge V',K')$. Then, the linear map $\varphi^*\colon (V')^*\inj V^*$ 
restricts to an embedding $\RR(\varphi)\colon \RR(V',K')\inj \RR(V,K)$ 
between the respective resonance varieties. 

After identifying the ring $S=\Sym(V)$ with the coordinate 
ring of $V^*$, we may view the support of the $S$-module 
$\Wey(V,K)$ as a subvariety of the affine space $V^*$. 
The next lemma connects this support variety to the resonance 
variety of the pair $(V\wedge V,K)$.

\begin{lemma}
\label{lem:supp}
Let $K\subseteq V\wedge V$ be a linear subspace.  Then:
\begin{enumerate}
\item \label{rr1}
Away from the origin, $\RR(V,K) = \supp (\Wey(V,K))$.
\item  \label{rr2}
$\dim_{\C} \Wey(V,K) < \infty $ if and only if $\RR(V,K)=\{0\}$. 
\end{enumerate}
\end{lemma}

\begin{proof}
Let $\delta_p(a) \colon \bigwedge^p V \to 
\bigwedge^{p-1} V$ be the evaluation of the $p$-th Koszul 
differential at a non-zero element $a\in V^*$.  Using 
formula \eqref{eq:wmod again} and the usual description 
of support in terms of Fitting ideals, we deduce that
\begin{equation} 
\label{eq:supp again}
a\in \supp (\Wey(V,K)) \same \dim_{\C} (\im(\delta_2(a)) /
\im(\delta_2(a) \circ  \iota)) \ge 1.
\end{equation}

Let $\delta_p^*(a)$ be the transpose of the matrix $\delta_p(a)$. 
Direct calculation shows that $\delta_2^*(a)\colon V^* \to V^* \wedge V^*$ 
is the map $b \mapsto a\wedge b$.  Hence, 
\begin{equation} 
\label{eq:res again}
a\in \RR(V,K) \same \dim_{\C}(\ker(\pi \circ \delta^*_2(a))/
 \ker(\delta_2^*(a))) \ge 1, 
\end{equation}
where $\pi\colon V^*\wedge V^* \to (V^*\wedge V^*)/K^{\perp}$ 
is the canonical projection. 

Upon identifying the matrix of $\pi$ with the transpose of the 
matrix of $\iota$, part \eqref{rr1} follows by comparing 
formulas \eqref{eq:supp again} 
and \eqref{eq:res again}. Finally, 
part \eqref{rr2} follows from part \eqref{rr1} by standard 
commutative algebra. 
\end{proof}

\subsection{Grassmannians and resonance}
\label{subsec:grass}

As usual, let $K\subseteq V\wedge V$ be a linear subspace.  
Setting $m=\dim K$, we may view $K$ as a point in 
the Grassmannian $\Grass_m(V\wedge V)$ of $m$-planes  
in $V\wedge V$, endowed with the Zariski topology. 

Now let $K^{\perp}\subseteq V^*\wedge V^*$ 
be the orthogonal complement. Then $\PP(K^{\perp})$ may 
be viewed as a codimension $m$ projective subspace in 
$\PP(V^*\wedge V^*)$.  

\begin{lemma}
\label{lem:r1vanish}
Let $\Grass_2(V^*) \inj \PP(V^*\wedge V^*)$ be the Pl\"ucker 
embedding. Then, 
\begin{equation*}
\label{eq:r1is0}
\RR(V,K)=\{0\} \same \PP(K^{\perp})\cap \Grass_2(V^*)=\emptyset.
\end{equation*}
\end{lemma}

\begin{proof}
Follows straight from the definition of resonance. 
\end{proof}

Now denote the dimension of $V$ by $n$, and fix an 
isomorphism $V=\C^n$. Consider the following 
subset of the Grassmannian of $m$-planes in 
$V\wedge V=\C^{\binom{n}{2}}$:
\begin{equation}
\label{eq:unm}
U_{n,m}=\big\{ K \in \Grass_m (V\wedge V)
\mid \RR(V,K)=\{0\}  \big\}.
\end{equation}

\begin{prop}
\label{prop:unm}
For any integer $m$ with $0\le m\le \binom{n}{2}$, where $n=\dim V$, 
the set $U_{n,m}$ is a Zariski open subset of $\Grass_m(V\wedge V)$.
\end{prop}

\begin{proof}
Let $\Grass^m(V^*\wedge V^*)$ be the Grassmannian of 
codimension $m$ planes in $V^*\wedge V^*$. In view of 
Lemma \ref{lem:r1vanish}, under the isomorphism 
$ \Grass_m (V\wedge V) \isom \Grass^m(V^*\wedge V^*)$, 
$K \leftrightarrow K^{\perp}$, the set $U_{n,m}$ corresponds 
to the set 
\begin{equation}
\label{eq:unm alt}
\mathcal{U}_{n}^{m} = \big\{ K^{\perp} \in \Grass^m(V^*\wedge V^*) 
\mid \PP(K^{\perp})\cap \Grass_2(V^*)=\emptyset  \big\}.
\end{equation}

Clearly, $\mathcal{U}_{n}^{m}$ is the complement in $\Grass^m(V^*\wedge V^*)$ 
to the set of codimension $m$ planes incident to $\Grass_2(V^*)$.  As is 
well-known (see e.g.~\cite{Har}), this latter set is Zariski closed. 
The desired conclusion follows at once.
\end{proof}

The next result identifies the threshold value of $m$ (as a function 
of $n$) for which the set $U_{n,m}$ is non-empty.

\begin{prop}
\label{prop:genvanish}
With notation as above,
\begin{equation*}
\label{eq:kn}
U_{n,m}\ne \emptyset \same m\ge 2n -3.
\end{equation*}
\end{prop}

\begin{proof}
The Grassmannian variety $\GG= \Grass_2(\C^n)$ may be viewed 
as a smooth, irreducible subvariety of the projective space 
$\PP=\PP\big(\C^{\binom{n}{2}}\big)$.  
Thus, if $W=V(f)$ is a hypersurface in $\PP$,  
then $\dim (\GG\cap W)\ge \dim ( \GG) -1$, with equality achieved 
if $f\notin I(\GG)$. 

Let $\K=\PP(K^{\perp})$ be a codimension $m$ projective 
subspace in $\PP$.  
Then 
\[
\dim (\GG\cap \K)\ge \dim \GG - \codim \K = 2(n-2) - m.
\]
Thus, if $m<2n-3$, then $\dim (\GG \cap  \K)\ge 0$, and so 
$\GG\cap \K\ne \emptyset$, showing that $K\notin U_{n,m}$. 

Now suppose $m\ge 2n -3$.  Then, there 
is a codimension $m$ projective subspace $\K\subset \PP$  
which intersects $\GG$ in the empty set.  
Thus, $K\in U_{n,m}$.
\end{proof}

\subsection{The graded pieces of a Koszul module}
\label{subsec:tor}
As before, let $K\subseteq V\wedge V$ be a subspace, 
and let $K^{\perp}\subseteq V^*\wedge V^*$ 
be its orthogonal complement. Set 
\begin{equation}
\label{eq:avk}
A(V,K) = \bigwedge V^* / \text{ideal}(K^{\perp}) .
\end{equation}
Clearly, $A=A(V,K)$ is a finite-dimensional, graded-commutative 
$\C$-algebra; moreover, $A$ is quadratic (i.e., it has a presentation 
with generators in degree $1$ and relations in degree $2$), and 
satisfies $a^2=0$, for all $a\in A^1$. 

Using a result of Fr\"oberg and L\"ofwall \cite[Theorem 4.1]{FL}, as recast 
in \cite[Proposition 2.3]{PS-artin}, as well as  \cite[Theorem 6.2]{PS-imrn}, 
we may reinterpret the (duals of the) graded pieces of 
the Koszul module $\Wey(V,K)$ in terms of the linear 
strand in an appropriate $\Tor$ module.  

\begin{prop}
\label{prop:weyq}
For all $q\ge 0$,
\begin{equation*}
\label{eq:weyq}
\Wey_{q}(V,K)^{*} \cong 
\Tor^{\bigwedge\! V^*}_{q+1}(A(V,K),\C)_{q+2}. 
\end{equation*}
\end{prop}

\begin{example}
\label{ex:chen raag}
Suppose the subspace $K\subseteq V\wedge V$ admits a monomial 
basis, that is to say, suppose there is a basis $\mathsf{V}=\{v_1,\dots , v_n\}$ 
for $V$ and a subset $\mathsf{E}\subseteq \binom{\mathsf{V}}{2}$ 
such that the set $\{v_i \wedge v_j \mid \{i,j\}\in  \mathsf{E}\}$ is a basis for $K$.
Then the algebra $A(V,K)$ is the exterior Stanley-Reisner ring associated 
to the graph $\Gamma$ with vertex set $\mathsf{V}$ and edge set 
$\mathsf{E}$.  Using Proposition \ref{prop:weyq}, we showed 
in \cite[Theorem 4.1]{PS-artin} that the Hilbert series of 
the graded module $\Wey(V,K)$ is given by
\begin{equation}
\label{eq:cut poly}
\sum_{q=0}^{\infty} \dim \Wey_q(V,K) t^{q+2} = 
Q_{\G}\left( \frac{t}{1-t} \right). 
\end{equation}

Here, $Q_{\G}(t)=\sum_{j\ge 2} c_j(\G) t^j$ 
is the ``cut polynomial" of $\G$, with coefficient   
$c_j(\G)$ equal to $\sum_{\W\subset \V\colon  \abs{\W}=j } 
\tilde{b}_0(\G_\W)$, where $\tilde{b}_0(\G_\W)$ is one less 
than the number of components of the full subgraph on $\W$. 
\end{example}

\subsection{A vanishing range}
\label{subsec:stable}
Now suppose the subspace $K$ belongs 
to the Zariski open set $U_{n,m} \subseteq \Grass_m(V\wedge V)$ 
from \eqref{eq:unm}.  Then, as we shall show in the next proposition,  
all the graded pieces of the Koszul module $\Wey(V,K)$ 
vanish beyond a range which depends only on the integers 
$n=\dim V$ and $m=\dim K$.

\begin{prop}
\label{prop:gen}
For each $m\ge 2n-3$, there is an integer 
$q=q(n,m)$ such that $\Wey_q(V,K)=0$, for every $K\in U_{n,m}$.
\end{prop}

\begin{proof}
For each $q\ge 0$, consider the following subset 
of the Grassmannian of $m$-planes in $V\wedge V$:
\begin{equation}
\label{eq:uqnm}
U^q_{n,m}= \big\{ K \in \Grass_m (V\wedge V)
\mid \Wey_q(V,K)=0 \big\}.
\end{equation}
By definition, the vector space $\Wey_q(V,K)$ is the 
cokernel of a matrix whose entries depend continuously 
on the $m$-plane $K\subset V\wedge V$. 
Thus, the set $U^q_{n,m}$ where those matrices have 
maximal rank is a Zariski open subset of $\Grass_m (V\wedge V)$. 

Now, Lemma \ref{lem:supp}\eqref{rr2} and definitions 
\eqref{eq:unm} and \eqref{eq:uqnm} together imply that 
\begin{equation}
\label{eq:unm bis}
U_{n,m}=\bigcup_{q\ge 0} U^q_{n,m}.
\end{equation}
Since the graded $\Sym (V)$-module $\Wey (V,K)$ is generated 
in degree zero, the union in \eqref{eq:unm bis} is an increasing union.  

Putting things together, we conclude that $U_{n,m}$ is 
an increasing union of Zariski open sets.  
Thus, there must exist an integer $q$ (depending on $n$ and $m$) 
such that $U^q_{n,m}=U_{n,m}$.  This finishes the proof.
\end{proof}

For instance, if $\binom{n}{2}-m\le 1$, then $q(n,m)=\binom{n}{2}-m$. 
In general, though, the determination of the integer $q(n,m)$ seems 
to be a challenging, yet interesting problem.  We will come back to this 
point in Section \ref{sect:weyman}.

\section{Some representation theory}
\label{sect:weights}

In this section, we analyze the decomposition into irreducible 
summands of the second exterior power of an irreducible 
representation of a complex, semisimple Lie algebra. 

\subsection{Semisimple Lie algebras}
\label{subsec:reps}
We start by reviewing some basic notions from \cite{Hu72}. 
Let $\g$ be a finite-dimensional, semisimple Lie algebra over $\C$. 
A choice of Cartan subalgebra, $\h\subset \g$, yields a root 
system for $\g$: by definition, this is the set $\Phi\subset \h^*$  
of eigenvalues of the adjoint action of $\h$ on $\g$. 

The Killing form defines an inner product, $(\,,)$ on $\h$.  It also sets 
up an isomorphism $\h \cong \h^*$, which gives by transport of structure 
an inner product on $\h^*$, denoted also by $(\,,)$.  
Given $\alpha, \beta \in \h^*$, write
\begin{equation}
\label{eq:ab}
\langle \alpha, \beta \rangle := \frac{2(\alpha, \beta)}{(\beta,\beta)}. 
\end{equation}
 
Inside the root system $\Phi$, fix a set of positive roots $\Phi^{+}$, 
so that $\Phi = \Phi^{+} \sqcup \Phi^{-}$. 
This choice determines a direct sum decomposition, 
$\g = \g^{-} \oplus \h \oplus \g^{+}$, where 
$\g^{\pm} = \bigcup_{\alpha\in \Phi^{\pm}} \g_{\alpha}$, 
and $\g_{\alpha}=\{x \in \g \mid [h,x]=\alpha(h)x, \forall h\in \h\}$.

Let $\Delta\subset \Phi^{+}$ be the corresponding set of simple roots. 
Note that $\h^*=\h^*_{\Q}\otimes_{\Q} \C$, where $\h^*_\Q$ is the 
$\Q$-vector space spanned by $\Delta$. Define 
the height function, $\height \colon \h^*_{\Q} \to \Q$, as the $\Q$-linear 
function given by $\height (\alpha) = 1$ for every $\alpha\in \Delta$. 

Inside the rational vector space $\h^*_{\Q}$, there are two 
subsets worth singling out:  the integral weight lattice, 
$\Lambda$,  defined as 
\begin{equation}
\label{eq:lambda}
\Lambda =\big\{ \lambda \in \h^*_{\Q} \mid  
\text{$\langle \lambda,\alpha\rangle \in \Z$, for all $\alpha\in \Delta$}\big\}, 
\end{equation}
and the positive cone, $C^+$, defined as 
\begin{equation}
\label{eq:cplus}
C^+ =\Big\{ \gamma \in \h^*_{\Q} \mid \text{
$\gamma=\sum_{\alpha\in \Delta} c_{\alpha} \alpha$, with 
$c_{\alpha} \in \Z_{\ge 0}$}\Big\}.
\end{equation}
Clearly, $\Phi^{+}\subset C^+$. 
The positive cone defines a partial order on the weight lattice 
$\Lambda$:  we say $\nu\preceq \tau$ if $\tau-\nu \in C^+$. 

\subsection{Irreducible representations}
\label{subsec:irreps}
Let $\Lambda^{+}\subset \Lambda$ be the set of 
dominant weights, defined by
\begin{equation}
\label{eq:lambdaplus}
\lambda \in \Lambda^{+}  \same 
\langle \lambda,\alpha\rangle \in \Z_{\ge 0}, \: \forall \alpha\in \Delta.
\end{equation}

This set parametrizes all finite-dimensional, 
irreducible representations of the Lie algebra $\g$:  for each 
dominant weight $\lambda\in \Lambda^{+}$, the corresponding  
representation is denoted by $V(\lambda)$.  
A non-zero vector $v\in V(\lambda)$ is said to be 
a maximal vector (of weight $\lambda$) if $\g^{+}\cdot v=0$.  
Such a vector is uniquely determined (up to non-zero scalars), 
and will be denoted by $v_{\lambda}$.

The dual representation, $V(\lambda)^*$, 
can be written as $V(\lambda)^* = V(\lambda^*)$, 
for a unique weight $\lambda^*\in \Lambda^{+}$.  
We will denote a maximal vector for $V(\lambda^*)$ 
by $v_{\lambda^*}$.

Given a finite-dimensional $\g$-module $U$,   
we associate to it two sets of weights, as follows.  

\begin{definition}
\label{def:vu}
The set of dominant weights occurring in $U$ is 
\[
\VV(U)  =\{ \mu \in  \Lambda^{+} \mid  \text{$U$ contains a  
summand isomorphic to $V(\mu)$}\}
\]
\end{definition}

Denote by $U_{\mu}$ the eigenspace of the $\h$-action 
on $U$, with eigenvalue $\mu$.  

\begin{definition}
\label{def:wu}
The set of weights occurring in $U$ is 
\[
\WW(U)  =\{ \mu \in  \Lambda \mid U_{\mu} \ne 0 \}.
\]
\end{definition}
It follows from the definitions that $\VV(U)\subseteq \WW(U)$.
When viewed as an $\h$-module by restriction, $U$ decomposes as 
\begin{equation}
\label{eq:decomp}
U= \bigoplus_{\mu \in \WW(U)} U_{\mu}. 
\end{equation}

\subsection{A weight test}
\label{subsec:2lamba}

Let $V(\lambda)$ be an irreducible $\g$-module corresponding 
to a dominant weight $\lambda$.  The next lemma provides a 
simple criterion for the existence of irreducible summands with 
certain prescribed weights in the second exterior power of $V(\lambda)$.  

\begin{lemma}
\label{lem:weights}
The representation $V(\lambda) \wedge V(\lambda)$ contains 
a direct summand isomorphic to $V(2\lambda -\beta)$, for some 
simple root $\beta$, if and only if $(\lambda,\beta)\ne 0$.  
When it exists, such a summand is unique.
\end{lemma}

\begin{proof}
First suppose $2\lambda -\beta$ is a dominant weight, 
for some $\beta\in \Delta$. Then $(2\lambda-\beta,\alpha) \ge 0$, 
for every $\alpha\in \Delta$. In particular, 
$(2\lambda-\beta, \beta)\ge 0$, which forces $(\lambda,\beta)\ne 0$, 
by the non-degeneracy of the Killing form. 

Conversely, suppose that $(\lambda,\beta)\ne 0$.  Then, 
by \cite[Ch.~VIII, \S 7, Exercice~17]{Bo}, the second exterior 
power of $V(\lambda)$ contains a unique direct summand 
of type $V(2\lambda -\beta)$. 
\end{proof}

\subsection{A decomposable maximal vector}
\label{subsec:max}

For every $\alpha\in \Phi^+$, there exist elements 
$x_\alpha, y_\alpha \in \g$ and $h_\alpha \in \h$ such that 
$\{x_\alpha, y_\alpha, h_\alpha\}\cong \sl_2(\C)$. 
We refer to \cite[Lemma 21.2]{Hu72} 
for basic commutation relations among those elements.  
In particular, $[x_\alpha,y_\beta]=0$, for all distinct 
$\alpha, \beta\in \Delta$.   It is readily seen that 
$\g^+$ is the subspace generated by 
$\{x_\alpha : \alpha\in \Phi^+\}$. 

Let $\lambda\in \Lambda^+$ be a dominant weight, and 
let $v_{\lambda}$ be a maximal vector for $V(\lambda)$. 
Given a simple root $\beta \in \Delta$, let $K(\lambda,\beta)$ 
be the $\g$-invariant subspace of $V(\lambda)\wedge V(\lambda)$ 
spanned by the vector $v_{\lambda}\wedge y_{\beta} v_{\lambda}$. 

\begin{lemma}
\label{lem:qqq2}
Let $\beta\in \Delta$ be a simple root, and let 
$V=V(\lambda)$ be an irreducible representation, with maximal 
vector $v_{\lambda}$.  If $(\lambda,\beta)\ne 0$, then 
$v_{\lambda}\wedge y_{\beta} v_{\lambda}$ 
is a maximal vector for $K(\lambda,\beta)$. 
In particular, $K(\lambda,\beta)$ is an irreducible $\g$-module, 
of highest weight $2\lambda -\beta \in \VV(V\wedge V)$. 
\end{lemma}

\begin{proof}
Set $\mu:=2\lambda -\beta$, and consider the vector 
$y_{\beta} v_{\lambda}\in V(\lambda)$. Since $x_{\beta}v_{\lambda}=0$, 
we have that 
\begin{equation}
\label{eq:serre}
x_{\beta}y_{\beta} v_{\lambda}= h_{\beta} v_{\lambda} 
= \lambda(h_{\beta}) v_{\lambda}  
=2 \frac{(\lambda,\beta)}{(\beta,\beta)} v_{\lambda}\ne 0,
\end{equation}
and thus $y_{\beta} v_{\lambda}\ne 0$.  Moreover, the vector 
$y_{\beta} v_{\lambda}$ has weight $\lambda- \beta \ne \lambda$. 
Thus, $v_{\lambda}$ and $y_{\beta} v_{\lambda}$ 
are linearly independent in $V(\lambda)$.  

Now set $k_{\mu}=v_{\lambda}\wedge y_{\beta} v_{\lambda}$.  
By the above, $k_{\mu}\ne 0$. Note that $k_{\mu}$ has weight 
$\lambda+(\lambda-\beta)=\mu$.
To check the remaining 
condition for maximality, that is, $\g^+ k_{\mu}=0$, it is 
enough to show that $x_\alpha k_{\mu}=0$ for all $\alpha\in \Delta$. 
If $\beta\ne \alpha$, then $[x_{\alpha},y_{\beta}]=0$, and so 
$x_{\alpha} k_{\mu} =v_{\lambda} \wedge 
y_{\beta}x_{\alpha} v_{\lambda}=0$. Moreover, 
$x_{\beta} k_{\mu} =v_{\lambda} \wedge 
x_{\beta} y_{\beta} v_{\lambda}=0$, by \eqref{eq:serre}. 

Thus, $k_{\mu}$ is a maximal vector (of weight $\mu$) for $K(\lambda,\beta)$.  
In other words, $K(\lambda,\beta)\cong V(\mu)$, and we conclude that 
$\mu \in \VV(V \wedge V)$.
\end{proof}

\section{Weights and resonance}
\label{sect:test}

In this section, we prove our main results concerning the 
resonance varieties associated to suitable representations 
of complex semisimple Lie algebras.

\subsection{A non-vanishing criterion}
\label{subsec:not zero}

We start with a test insuring that the resonance variety 
associated to a $\g$-invariant subspace $K\subset V\wedge V$ 
does not vanish. 

Let $V=V(\lambda)$ be an irreducible $\g$-module, corresponding 
to a dominant weight $\lambda\in \Lambda^{+}$, and let $v_{\lambda}$ 
be a maximal vector in $V$. 
Given a simple root $\beta \in \Delta$, recall that $K(\lambda,\beta)$ 
denotes the $\g$-invariant subspace of $V\wedge V$ 
spanned by the vector $v_{\lambda}\wedge y_{\beta} v_{\lambda}$. 

\begin{prop}
\label{prop:non-vanish}
Let $K\subset V\wedge V$ be a $\g$-invariant 
subspace.  Suppose $K^{\perp}\supseteq K(\lambda^*,\beta)$, 
for some $\beta\in \Delta$ with $(\lambda^*,\beta)\ne 0$. 
Then $v_{\lambda^*}$ belongs to $\RR(V,K)$. 
\end{prop}

\begin{proof}
Since $(\lambda^*,\beta)\ne 0$, Lemma \ref{lem:qqq2} guarantees  
that the $\g$-module $K(\lambda^*,\beta)$ is irreducible, and 
has maximal vector $v_{\lambda^*}\wedge y_{\beta} v_{\lambda^*}$, 
of weight $2\lambda^*-\beta$.  Thus,  
$v_{\lambda^*}\wedge y_{\beta} v_{\lambda^*}$ is a non-zero 
element in $K^{\perp}$.  
Therefore, $v_{\lambda^*}\in \RR(V(\lambda),K)$.
\end{proof}

\subsection{Maximal vectors in resonance varieties}
\label{subsec:lie}

We record now a basic symmetry property enjoyed by 
resonance varieties.

\begin{lemma}
\label{lem:inv res}
Let $G$ be a group, and let $\rho\colon G\to \GL(V)$ be a 
finite-dimensional representation.  Let $K\subset V\wedge V$ 
be a $G$-invariant subspace.  Then, the corresponding 
resonance variety, $\RR(V,K)$, is a $G$-invariant subset 
of $V^*$.  
\end{lemma}

\begin{proof}
Let $a\in V^*$ be a non-zero element of $\RR(V,K)$. 
By definition, this means there is an element $b\in V^*$ 
such that $a\wedge b\ne 0$ and $a\wedge b\in K^{\perp}$. 

Now let $g$ be an element of $G$. Since $G$ 
acts diagonally on $V^*\wedge V^*$, and since $K^{\perp}$ 
is a $G$-invariant subspace of $V^*\wedge V^*$, 
we have that $ga\wedge gb\ne 0$ and $ga\wedge gb\in K^{\perp}$. 
This shows that $ga\in \RR(V,K)$, and we are done. 
\end{proof} 

Next, we recall a lemma from \cite{DP}.

\begin{lemma}[\cite{DP}]
\label{lem:borel}
Let $G$ be a complex, semisimple, linear algebraic group, with 
Lie algebra $\g$.   Let $U$ be an irreducible, rational 
representation of $G$.  If $\RR$ is a Zariski-closed, $G$-invariant 
cone in $U$, and if $\RR \ne \{ 0\}$, then $\RR$ contains a 
maximal vector for the $\g$-module $U$. 
\end{lemma}

The proof of this lemma is based on the classical Borel fixed 
point theorem (cf.~\cite{Hu75}), which insures that the action 
of the Borel subgroup $B$ with Lie algebra $\h \oplus \g^{+}$ 
has a fixed point in the projectivization $\PP(\RR) \subseteq \PP (U)$. 

Using the above two lemmas, and some standard facts about 
Lie groups and Lie algebras from Serre's monograph \cite{Se}, 
we can now pin down a maximal vector in each non-zero 
resonance variety associated to an irreducible $\g$-representation. 

\begin{prop}
\label{prop:inf borel}
Let $\g$ be a complex, semisimple Lie algebra, let $V$ 
be an irreducible $\g$-module, and let $K\subseteq V\wedge V$ 
be a $\g$-submodule. If the resonance variety $\RR=\RR(V,K)$ 
is not equal to $\{ 0\}$, then $\RR$ contains a maximal vector 
for the $\g$-module $V^*$. 
\end{prop}

\begin{proof}
As shown in \cite{Se}, there is a unique simply-connected, complex, 
semisimple, linear algebraic group $G$ such that $\Lie(G)=\g$.
Furthermore, the vector space $V$ supports a 
rational, irreducible representation of $G$, such that the associated 
infinitesimal representation of $\Lie(G)$ coincides with the given 
$\g$-module structure on $V$.

It follows that the subspace $K\subseteq V\wedge V$ is $G$-invariant.  
Hence, by Lemma \ref{lem:inv res}, the set $\RR=\RR(V,K)$ is also 
$G$-invariant.   Of course, $\RR$ is a homogeneous variety, 
and thus a conical subset of $V^*$; furthermore, $\RR\ne \{0\}$, 
by assumption.   Thus, all the hypothesis of Lemma \ref{lem:borel} 
have been verified, and we conclude that $\RR$ contains a maximal 
vector for $V^*$.
\end{proof}

\subsection{Decomposing a maximal vector}
\label{subsec:decomp}

The following lemma generalizes Lemmas 3.5 from \cite{DP} 
and  9.6 from \cite{PS-johnson}. Although the proof 
follows in rough outline the proofs of those two previous results, 
we include here full details, for the reader's convenience. 

\begin{lemma}
\label{lem:engel}
Let $V=V(\lambda)$ be an irreducible $\g$-module, and 
let $v_{\lambda^*}$ be a maximal  vector for $V^*$.   
Let $K$ be a  $\g$-invariant 
subspace of $V\wedge V$. 
Suppose $v_{\lambda^*}$ belongs to $\RR (V,K)$.   
Then, we may  find an element $w\in V^*$ such that 
if we set $u=v_{\lambda^*} \wedge w$, we have 
$u\in K^{\perp}$, $u\ne 0$, and $\g^+\cdot u=0$. 
\end{lemma}

\begin{proof}
Let $\ell\colon V^*\to V^* \wedge V^*$ denote left-multiplication 
by the element $v_{\lambda^*}$ in the exterior algebra on $V^*$. 
By definition of resonance, the vector $v_{\lambda^*}$ belongs to 
$\RR(V,K)$ if and only if $\im (\ell) \cap K^{\perp} \ne 0$. 
The Lie algebra $\g^{+}$ annihilates the vector $v_{\lambda^*}$; 
hence, the linear map $\ell$ is $\g^{+}$-equivariant. 
Thus, all we need to show is that the $\g^{+}$-module 
$\im (\ell) \cap K^{\perp} $ contains a non-zero vector 
$u$ annihilated by $\g^{+}$.

This claim follows from Engel's theorem (cf.~\cite[Theorem 3.3]{Hu72}), 
provided each element of $\g^{+}$ acts nilpotently on $K^{\perp}$.
To verify this nilpotence property, we may assume $K^{\perp}$ 
is $\g$-irreducible, the general case following easily from this one. 

The Lie algebra $\g^+$ decomposes into a direct sum 
of $1$-dimensional vector spaces, each one spanned by an element 
$x_{\alpha}$, with $\alpha$ running through the set of positive roots, 
$\Phi^{+}$.   
Let $\mu$ be the dominant weight corresponding to $K^{\perp}$. 
Let $K^{\perp}_{\nu}$ be a non-trivial weight space for $K^{\perp}$, 
and let $\alpha_1,\dots ,\alpha_r\in \Phi^{+}$.  
Then 
\begin{equation}
\label{eq:kpnew}
x_{\alpha_1}\cdots x_{\alpha_r}  (K^{\perp}_{\nu}) \subseteq  K^{\perp}_{\nu'}, 
\end{equation}
where $\nu'=\nu + \sum_{i=1}^r \alpha_i$.
Suppose $K^{\perp}_{\nu'} \ne 0$; 
then, by maximality of $\mu$, we may write $\nu'=\mu-\beta$, 
for some element $\beta$ in the positive cone, see \cite[Theorem 20.2(b)]{Hu72}. 
Hence, $\sum_{i=1}^r \alpha_i = \mu- \nu -\beta$.  Since $\alpha_i \in \Phi^+$
and $\beta \in C^+$, we find that
\begin{equation}
\label{eq:height}
r\le \height \Big(\sum_{i=1}^r \alpha_i \Big) \le \height(\mu-\nu).
\end{equation}

Choosing an integer $r$ strictly greater than $\height(\mu-\nu)$, 
for all $\nu \in \WW (K^{\perp})$, formula \eqref{eq:height} guarantees 
that $K^{\perp}_{\nu'}=0$.  From \eqref{eq:kpnew}, it follows 
that each element of $\g^{+}$ acts nilpotently on $K^{\perp}$.  
This completes the proof.
\end{proof}

\subsection{A vanishing criterion}
\label{subsec:simple roots}

We now provide a test insuring that the resonance variety 
associated to a $\g$-invariant subspace $K\subset V\wedge V$ 
does vanish. 

\begin{theorem}
\label{thm:root vanish}
Let $V$ be an irreducible $\g$-module  
of highest weight $\lambda$, and let $K$ be a $\g$-submodule of 
$V\wedge V$. Suppose that $2\lambda^* -\mu\notin \Delta$, for 
any $\mu\in \VV(K^{\perp})$.  Then $\RR(V,K)=\{0\}$.  
\end{theorem}

\begin{proof}
Suppose $\RR=\RR(V,K)$ is non-zero.  Then, by 
Proposition \ref{prop:inf borel}, the variety $\RR\subset V^*$ 
contains a maximal vector $v_{\lambda^*}$ for the 
irreducible $\g$-module $V^*=V(\lambda^{*})$.  We 
will use this fact, together with our assumption 
on the weight $\lambda^*$, to derive a contradiction.

Lemma \ref{lem:engel} yields an element $w\in V^*$ such that, if we set 
$u=v_{\lambda^*} \wedge w$, then
\begin{equation}
\label{eq:uk}
u\in K^{\perp}, \quad u\ne 0, \quad \g^+\cdot u=0.
\end{equation}

Let $V^*=\bigoplus_{\nu\in \WW(V^*)} V^*_{\nu}$ be the weight decomposition 
of $V^*$ under the action of $\h$; accordingly, write $w=\sum_{\nu} w_{\nu}$, 
where each $w_{\nu}$ belongs to $V^*_{\nu}$, and thus 
has weight $\nu$.  Finally, set 
\begin{equation}
\label{eq:uknu}
u_{\lambda^*+\nu}=v_{\lambda^*} \wedge w_{\nu}.
\end{equation}

Note that $u_{\lambda^*+\nu}$ has weight $\lambda^*+\nu$. 
Hence, $u=\sum_{\nu} u_{\lambda^*+\nu}$ is the weight decomposition 
of the vector $u\in K^{\perp}$.  Since $K^{\perp}$ is a $\g$-submodule 
of $V^*\wedge V^*$, it follows that each $u_{\lambda^* + \nu}$ belongs to 
$K^{\perp}$.  

Now, for each $\alpha\in \Phi^+$, we have that 
$x_{\alpha} u=0$; hence, $x_{\alpha} u_{\lambda^*+\nu}=0$ 
for all $\nu$, by uniqueness of the weight decomposition. 
Finally, since $u\ne 0$, one of the components of $u$ must 
be non-zero.  Therefore, there is an index $\nu$ such that 
\begin{equation}
\label{eq:uknu bis}
u_{\lambda^*+\nu}\in K^{\perp}, \quad u_{\lambda^*+\nu}\ne 0, 
\quad \g^+\cdot u_{\lambda^*+\nu}=0.
\end{equation}

Set $\mu:=\lambda^*+\nu$.  By \eqref{eq:uknu bis}, then, 
\begin{equation}
\label{eq:mln}
\mu\in \VV(K^{\perp}).
\end{equation}

For any simple root $\alpha\in \Delta$, we have that 
\begin{equation}
\label{eq:xu}
0=x_{\alpha} u_{\lambda^*+\nu} = 
x_{\alpha} ( v_{\lambda^*} \wedge w_{\nu} ) = 
v_{\lambda^*}\wedge x_{\alpha} w_{\nu},
\end{equation}
and thus $x_{\alpha} w_{\nu} \in \C\cdot v_{\lambda^*}$. 
Suppose there is a simple root $\beta$ such that 
$x_{\beta} w_{\nu}\ne 0$.  A simple weight inspection 
reveals that 
\begin{equation}
\label{eq:nal}
\beta+\nu=\lambda^*.
\end{equation}
Putting together \eqref{eq:mln} and \eqref{eq:nal}, we obtain 
that $2\lambda^*-\mu=\beta\in \Delta$, with $\mu\in \VV(K^{\perp})$, 
thereby contradicting our hypothesis. Therefore, $x_{\alpha} w_{\nu}= 0$, 
for all $\alpha\in \Delta$, i.e., $\g^{+}\cdot w_{\nu}=0$.  

On the other hand, we also know that $w_{\nu}\ne 0$; thus, 
$w_{\nu}$ is a maximal vector in $V^*=V(\lambda^*)$.  
Hence, $w_{\nu}\in \C\cdot v_{\lambda^*}$. 
Therefore, $u_{\lambda^*+\nu}=v_{\lambda^*} \wedge w_{\nu}=0$, 
contradicting \eqref{eq:uknu bis}. Thus, $\RR=\{0\}$.
\end{proof}

\begin{corollary}
\label{cor:mult free}
Let $V$ be an irreducible $\g$-module  
of highest weight $\lambda$,  and set $V^*=V(\lambda^*)$. 
Let $K$ be a $\g$-submodule of $V\wedge V$.   
Then, the following conditions are equivalent:
\begin{enumerate}
\item \label{mf1}
$\RR(V,K)=\{0\}$.
\item  \label{mf2}
$\dim_{\C} \Wey(V, K) < \infty$.
\item  \label{mf3}
$2\lambda^* - \beta \notin \VV(K^{\perp})$, for all $\beta\in \Delta$ 
such that $(\beta,\lambda^*)\ne 0$.
\end{enumerate}
\end{corollary}

\begin{proof}
The equivalence \eqref{mf1} $\Leftrightarrow$ \eqref{mf2} follows 
from Lemma \ref{lem:supp}. The implication \eqref{mf3} 
$\Rightarrow$ \eqref{mf1} follows from 
Theorem \ref{thm:root vanish} and Lemma \ref{lem:weights}.

For the implication \eqref{mf1} $\Rightarrow$ \eqref{mf3},
set $\mu:= 2\lambda^* - \beta$, and suppose 
$\mu \in \VV(K^{\perp})$, for some $\beta\in \Delta$ 
such that $(\beta,\lambda^*)\ne 0$.  
In this case, there is a submodule $W\subseteq K^{\perp}$
isomorphic to $V(\mu)$. On the other hand, the submodule 
$K(\lambda^*, \beta) \subseteq V^*\wedge V^*$ is also 
isomorphic to $V(\mu)$, by Lemma \ref{lem:qqq2}. 
By the uniqueness property from \cite[Ch.~VIII, \S 7, Exercice~17]{Bo},
$K(\lambda^*, \beta)=W$.  
Hence, by Proposition \ref{prop:non-vanish}, 
$v_{\lambda^*}$ is a non-zero vector in $\RR(V,K)$,
and we are done.
\end{proof}

\section{Weyman modules}
\label{sect:weyman}

In this section, we treat in detail the case when $\g=\sl_2(\C)$.  
The representation theory of this simple Lie algebra is 
of course classical; for a thorough treatment, we refer 
to Fulton and Harris \cite{FH}.

The dual Cartan subalgebra, $\h^*$, is spanned 
by functionals $t_1$ and $t_2$ (the dual coordinates on the 
subspace of diagonal $2\times 2$ complex matrices), 
subject to the single relation $t_1+t_2=0$. 
The set $\Phi^{+}=\Delta$ consists of a single simple root, 
$\beta=t_1-t_2$.  The defining ($2$-dimensional) representation 
is $V(\lambda_1)$, where $\lambda_1=t_1$. Moreover, all the 
finite-dimensional, irreducible representations of $\sl_2(\C)$ 
are of the form 
\begin{equation}
\label{eq:vn}
V_n=V(n\lambda_1)=\Sym_n(V(\lambda_1)),
\end{equation} 
for some $n\ge 0$. Clearly, $\dim V_n =n+1$.  
Furthermore, $V_n^*=V_n$; in other words, all irreps of 
$\sl_2(\C)$ are self-dual.

For each $n\ge 1$, the second exterior power of $V_n$ 
decomposes into irreducibles, according to the Clebsch-Gordan rule:
\begin{equation}
\label{eq:sl2 wedge}
V_n \wedge V_n = \bigoplus_{j\ge 0} V_{2n-2-4j}.
\end{equation}
In particular, all summands in \eqref{eq:sl2 wedge} occur 
with multiplicity $1$, and $V_{2n-2}$ is always one of those 
direct summands. 

\begin{prop}
\label{prop:findim}
Let $K$ be an $\sl_2(\C)$-submodule 
of $V_n \wedge V_n$.  The following are equivalent:
\begin{enumerate}
\item The variety $\RR(V_n,K)$ consists only of $0\in V_n^*$. 
\item The $\C$-vector space $\Wey(V_n, K)$ is finite-dimensional. 
\item The representation $K$ contains $V_{2n-2}$ as a direct summand.
\end{enumerate}
\end{prop}

\begin{proof}
Let us apply Corollary \ref{cor:mult free} to this situation. 
We have $\lambda^*=\lambda= n t_1$ and $\beta=t_1-t_2$;  
thus, $(\lambda^*, \beta)\ne 0$ and 
$2\lambda^*-\beta = (2n-2) \lambda_1$.  
The desired conclusion follows at once.
\end{proof}

Following Eisenbud \cite{Ei}, let us single out an important 
particular case of the above construction. 

\begin{definition}
\label{def:wm}
For each $n\ge 1$, the corresponding {\em Weyman module}\/ is 
$W(n)=\Wey(V_n$, $V_{2n-2})$, 
viewed as a graded module over the polynomial ring $\Sym(V_n)$.
\end{definition}

Note that $V_{2n-2}$ belongs to the set $U_{n+1,2(n+1)-3}$ 
from \eqref{eq:unm}, and thus, it is in the critical range identified 
in Proposition \ref{prop:genvanish}.

The first part of the next corollary recovers an assertion from \cite{Ei}.

\begin{corollary}
\label{cor:fw}
For each $n\ge 1$, the following hold. 
\begin{enumerate}
\item \label{w1}
The Weyman module $W(n)=\Wey(V_n, V_{2n-2})$ has finite 
dimension as a $\C$-vector space.
\item \label{w2}
Given an $\sl_2(\C)$-submodule $K\subset V_n\wedge V_n$, the 
corresponding Koszul module $\Wey(V_n$, $K)$ is finite-dimensional 
over $\C$ if and only if $\Wey(V_n,K)$ is a quotient of $W(n)$.
\end{enumerate}
\end{corollary}

As explained by Eisenbud in \cite{Ei}, Weyman showed that 
the vanishing of $W_{n-2}(n)$, for all $n\ge 1$,  
implies the generic Green Conjecture on free resolutions 
of canonical curves.

\section{Alexander invariants and resonance varieties of groups}
\label{sect:alexinv}

\subsection{The Alexander invariant of a group}
\label{subsec:alex inv}

Let $G$ be a group, and let $(x,y)=xyx^{-1}y^{-1}$ denote the 
group commutator. The derived subgroup, $G'=(G,G)$, is a 
normal subgroup; the quotient group, $G_{\ab}=G/G'$, is the  
maximal abelian quotient of $G$.  Also let $G''=(G',G')$ 
be the second derived subgroup; then $G/G''$ is the maximal 
metabelian quotient of $G$.  

The abelianization map, $\ab\colon G\surj G_{\ab}$, factors 
through $G/G''$,  yielding an exact sequence, 
$0 \to G'/G'' \to G/G'' \to G_{\ab}\to 0$. 
Conjugation in $G/G''$ naturally makes the abelian group 
$G'/G''$ into a module over the group ring $\Z[G_{\ab}]$.   
Following W.~Massey \cite{Ma}, we call the complexification 
of this module, 
\begin{equation} 
\label{eq:alex inv}
B(G) := (G'/G'') \otimes \C= H_1(G',\C),
\end{equation}
the {\em Alexander invariant}\/ of $G$. 
By construction, $B(G)$ is a module over the group algebra 
$R=\C[G_{\ab}]$, with module structure given by 
$\bar{h}\cdot \bar{g}= \overline{hgh^{-1}}$, for elements 
$\bar{h}\in G/G'$, represented by $h\in G$, and 
$\bar{g}\in G'/G''$, represented by $g\in G'$.

As a slight variation on this construction, let 
$\abf\colon G\to G_{\abf}$ be the projection 
to the maximal torsion-free 
abelian quotient, $G_{\abf}=G_{\ab}/\Tors(G_{\ab})$,   
and define the {\em reduced Alexander invariant}\/ to be 
the vector space 
\begin{equation} 
\label{eq:free alex inv}
\BB(G)=H_1(\ker(\abf),\C), 
\end{equation}
viewed as a module over the group algebra 
$\widetilde{R}=\C[G_{\abf}]$.   We then have an epimorphism  
$B(G) \to \BB(G)$, equivariant with respect to the canonical 
projection $R\to \widetilde{R}$. 

\subsection{Infinitesimal Alexander invariant}
\label{subsec:bbis}

For the rest of this section, we will assume $G$ is a finitely 
generated group.  Set $V=H_1(G,\C)$, and identify the 
dual vector space, $V^*$, with $H^1(G,\C)$.   
Let $\cup_G\colon V^*\wedge V^*\to H^2(G,\C)$ be 
the cup-product map.  Its kernel, $K^{\perp}\subset V^*\wedge V^*$, 
is the orthogonal complement to a linear subspace 
$K\subset V\wedge V$; put another way, $K=\im (\partial_G)$, 
where $\partial_G\colon H_2 (G, \C) \to V\wedge V$
is the transpose of $\cup_G$.

Let us define the {\em infinitesimal Alexander invariant}\/ of $G$ to  
be the $S$-module 
\begin{equation} 
\label{eq:inf alex}
\B(G)=\Wey(V, K),
\end{equation} 
where $S=\Sym(V)$. It is readily seen that this definition agrees 
with the one from \cite{PS-imrn}, modulo a degree shift by $2$. 

Now let $I$ be the augmentation ideal of the ring $R=\C[G_{\ab}]$. 
The powers of $I$ define a filtration on the Alexander invariant of $G$;
the associated graded object, 
\begin{equation}
\label{eq:hatbg}
\gr B(G)= \bigoplus_{q\ge 0} I^q B(G)/I^{q+1} B(G), 
\end{equation}
may be viewed as a graded module over the ring $\gr R = S$.   
Similar considerations apply to the 
$S$-module $\gr \widetilde{B}(G)$.

Work from \cite{Ma, PS-imrn, PS-johnson} 
as well as \cite[Proposition 2.4]{DHP}, implies that
\begin{equation}
\label{prop:grb}
\dim_{\C} \gr_q B(G) = \dim_{\C} \gr_q \BB (G) 
\le \dim_{\C} \B_q(G), 
\end{equation}
for all $q\ge 0$.

\subsection{The formal situation}
\label{subsec:formal}
Now suppose the (finitely generated) group $G$ is {\em $1$-formal}, 
i.e., its Malcev Lie algebra $\m(G)$ is the completion of a quadratic 
Lie algebra (we refer to \cite{PS-imrn} for more details on this notion).  
Then, it follows from \cite{DPS-duke} that 
\begin{equation}
\label{eq:bbcompl}
\gr B(G) \cong \B(G),
\end{equation}
as graded vector spaces.

Recall that an $R$-module $M$ is said to be nilpotent if 
$I^q\cdot M=0$, for some $q\ge 1$. Clearly, if $M$ is 
nilpotent, then $M=\gr M$, as vector spaces.
The following corollary is now immediate.

\begin{corollary}
\label{cor:formal}
Suppose $G$ is a $1$-formal group. Then:
\begin{enumerate}
\item \label{f1}
$\gr B(G) =  \gr \BB (G) =  \B(G)$, as graded 
vector spaces.  
\item \label{f2}
If $B(G)$ is nilpotent, 
then $\dim_{\C} B(G)= \dim_{\C} \B(G)<\infty$.
\item \label{f3}
If $\BB(G)$ is nilpotent, 
then $\dim_{\C} \BB(G)= \dim_{\C} \B(G)<\infty$.
\end{enumerate}
\end{corollary}

\subsection{Resonance varieties of groups}
\label{subsec:res gp}
As before, let $G$ be a finitely generated group, with 
$V^*=H^1(G,\C)$ and $K^{\perp} =\ker(\cup_G)\subset V^*\wedge V^*$. 
The resonance variety of $G$ is then defined as 
\begin{equation}
\label{eq:res g}
\RR(G)=\RR(V,K).
\end{equation}
This definition coincides (at least away from the origin)  
with the usual definition of the first resonance variety, $\RR^1_1(G,\C)$.

\begin{prop}
Let $V$ be an $n$-dimensional $\C$-vector space.  
\begin{enumerate}
\item \label{q1}
For any linear 
subspace $K\subseteq V\wedge V$ defined over $\Q$, there 
is a finitely presented, commutator-relators group $G$ with 
$V^*=H^1(G,\C)$ and $K^{\perp} =\ker(\cup_G)$. 
\item \label{q2}
If $m\ge 2n-3$, the Zariski open set $U_{n,m}=\{ K \in \Grass_m (V\wedge V)
\mid \RR(V,K)=\{0\}  \}$ contains a rational $m$-plane.
\end{enumerate}
\end{prop}

\begin{proof}
To prove \eqref{q1}, 
fix a basis $e_1,\dots , e_n$ for $V$, and pick a basis 
$v_1,\dots, v_m$ for $K$, with entries 
$v_k=\sum_{1\le i<j\le n} c^k_{i,j} e_i \wedge e_j$  
having integral coefficients. Consider the group  with presentation 
$G=\langle x_1,\dots , x_n \mid r_1,\dots, r_m\rangle$, 
where $r_k=\prod_{1\le i<j\le n} (x_i,x_j)^{c^{k}_{i,j}}$.  
A standard Fox calculus computation shows that $K=\im (\partial_G)$.

To prove \eqref{q2}, note that the rational points in $\C^{m\binom{n}{2}}$ 
form a Zariski-dense subset.  
Hence, the conclusion follows from 
Proposition \ref{prop:genvanish}.
\end{proof}

As an immediate corollary, we see that the bound from 
Proposition \ref{prop:genvanish} is sharp, even for $m$-planes  
coming from finitely presented groups. 

\begin{corollary}
\label{cor:zerores gp}
For each $n\ge 2$, there is a finitely presented group $G$ such that  
$b_1(G)=n$, $\codim(\ker(\cup_G))=2n-3$, and $\RR(G)=\{0\}$. 
\end{corollary}

\begin{example}
\label{ex:n=4}
For $n=4$, consider the group 
\[
G=\langle x_1,\dots ,x_4 \mid (x_1,x_2), \, 
(x_2,x_3), \,  (x_3,x_4), \,  (x_1,x_4), \,  
(x_1,x_3)\cdot (x_2,x_4)\rangle. 
\]
In Pl\"ucker coordinates, the Grassmannian $\Grass_2(\C^4)\subset \PP(\C^6)$ 
is given by the equation $p_{12}p_{34}-p_{13}p_{24}+p_{23}p_{14}=0$, while 
the plane $\PP(K^{\perp})=\PP(\ker(\cup_G))\in \PP(\C^6)$ is given by  
$p_{12}=p_{23}=p_{34}=p_{14}=p_{13}+p_{24}=0$.  
Clearly, $\PP(K^{\perp})\cap \Grass_2(\C^4)=\emptyset$; 
thus, by Lemma \ref{lem:r1vanish}, $\RR(G)=\{0\}$.
\end{example}

Finally, recall that the {\em deficiency}\/ of a finitely presented group 
$G$, written $\df(G)$, is the supremum of the difference between 
the number of generators and the number of relators, taken 
over all finite presentations of $G$. 

\begin{corollary}
\label{cor:def}
Let $G$ be a finitely presented group.  Suppose $\RR(G)=\{0\}$. 
Then $\df(G)\le 3-b_1(G)$.
\end{corollary}

\begin{proof}
Let $G=\langle x_1,\dots, x_m\mid r_1, \dots, r_q\rangle$ 
be a finite presentation.  Set $n=b_1(G)$  and $k=\dim(\im(\partial_G))$. 
Then $k\le b_2(G) \le q-m+n$. 
From the hypothesis and Proposition \ref{prop:genvanish}, 
we have that $2n-3\le k$.  Thus, $2n-3 \le q-m+n$, and so $m-q \le 3-n$.  
The conclusion follows.
\end{proof}

\section{Torelli groups and vanishing resonance}
\label{sect:torelli res} 

\subsection{Torelli groups of free groups}
\label{subsec:torelli free} 

Let $F_n$ be the free group on $n$ generators. Identify the group 
$H=(F_n)_{\ab}$ with $\Z^n$  and the group $\Aut(\Z^n)$ 
with $\GL_n(\Z)$.   As is well-known, the Torelli group, 
$\OA_n=\T_{F_n}$, is finitely generated.  Furthermore, the 
canonical homomorphism $\Out(F_n)\to \GL_n (\Z)$ is surjective; 
thus, there is a natural $\GL_n (\Z)$-action on the homology 
groups of $\OA_n$. 

Work of Andreadakis, Cohen--Pakianathan, Farb and Kawazumi 
(see \cite{Pe}) shows that the action by restriction 
of $\SL_n (\Z)$ on $V=H_1(\OA_n,\C)$ extends to a rational, 
irreducible representation of the simply-connected, 
complex, semisimple, linear algebraic group $\SL_n(\C)$, 
and thus, of the Lie algebra $\sl_n(\C)$, which can be 
explicitly identified.  

Following Fulton and Harris \cite{FH}, denote by 
$t_1,\dots, t_n$ the dual coordinates of
the diagonal matrices from $\gl_n (\C)$, and set  
$\lambda_i =t_1+ \cdots+t_i$, for $1\le i\le n$.
Let $\h_n$ be the (diagonal) Cartan subalgebra of $\sl_n (\C)$; 
then $\h_n^* =\C\text{-span}\, \{ t_1,\dots, t_n \}/ \C\cdot \lambda_n$. 
The standard set of positive roots is $\Phi^{+}=\{ t_i-t_j \mid 1\le i<j \le n \}$, 
while the set of simple roots is $\Delta=\{ t_i - t_{i+1} \mid 
1\le i \le  n-1\}$.
The finite-dimensional irreducible representations of $\sl_n (\C)$ 
are parametrized by tuples $(a_1,\dots, a_{n-1})$ of non-negative integers;   
to such a tuple, there corresponds an irreducible representation 
$V(\lambda)$, with highest weight $\lambda =\sum_{i=1}^{n-1} a_i \lambda_i$. 

The following 
theorem recovers a result from \cite{PS-johnson}.

\begin{theorem}
\label{thm:rroan}
For each $n\ge 4$, the resonance variety 
$\RR(\OA_n)$ vanishes. 
\end{theorem}

\begin{proof}
Let $V^*= H^1 (\OA_n, \C)$, and let $K^{\perp}\subset V^*\wedge V^*$ 
be the kernel of the cup-product map in degree $1$.   
Set  $\lambda=\lambda_2 +\lambda_{n-1}$, so that 
$\lambda^*= \lambda_1 +\lambda_{n-2}$, and 
also set $\mu= \lambda_1 +\lambda_{n-2} +\lambda_{n-1}$.  As 
shown by Pettet in \cite{Pe},  $V^*=V(\lambda^*)$ and 
$K^{\perp}=V(\mu)$, as $\sl_n (\C)$-modules.  

It is immediate to see that 
$2\lambda^*-\mu=t_1-t_{n-1}$ is not a simple root.  
It follows from Theorem \ref{thm:root vanish}
that $\RR(V,K)=\{0\}$.
\end{proof}

\begin{remark}
\label{rem:n3}
When $n=3$, the above proof breaks down, since 
$t_1-t_{2}$ is a simple root.  In fact, 
$K^{\perp} = V^* \wedge V^*$ in this case, 
and so $\RR(V,K)=V^*$. 
\end{remark}

\begin{corollary}
\label{cor:boan}
For each $n\ge 4$, let  $V=V( \lambda_2 +\lambda_{n-1})$ and let
$K^{\perp} = V(\lambda_1+\lambda_{n-2}+\lambda_{n-1}) \subset V^*\wedge V^*$
be the Pettet summand, as above.  Then $\dim \Wey(V,K) <\infty$ and 
$\dim \gr_q B(\OA_n) \le \dim \Wey_q(V,K)$, for 
all $q\ge 0$.
\end{corollary}

\begin{proof}
Follows from Lemma \ref{lem:supp} 
and inequality \eqref{prop:grb}.
\end{proof}

\subsection{Torelli groups of surfaces}
\label{subsec:torelli surf} 

Let $\Sigma_g$ be a Riemann surface of genus $g$, 
and let $\T_g=\T_{\pi_1(\Sigma_g)}$ be the associated 
Torelli group.   For $g \le 1$, the group $\T_g$ is trivial, 
while for $g=2$, it is not finitely generated.  
So we will assume from now on that $g\ge 3$, in 
which case it is known that $\T_g$ is finitely generated. 

The canonical homomorphism $\Out^+(\pi_1(\Sigma_g))\to \Sp_{2g} (\Z)$ 
is surjective; thus, there is a natural $\Sp_{2g} (\Z)$-action 
on the homology groups of $\T_g$. Work of D.~Johnson \cite{J3}
shows that the action 
of $\Sp_{2g} (\Z)$ on $V=H_1(\T_g,\C)$ extends to a rational, 
irreducible representation of the simply-connected, 
complex, semisimple, linear algebraic group $\Sp_{2g}(\C)$, 
and thus, of the Lie algebra $\sp_{2g}(\C)$.  

Let $\h \subset \sp_{2g}(\C)$  be the 
Cartan subalgebra of diagonal matrices, 
and let $t_1,\dots, t_g$ be the standard   
basis of $\h^*$, orthogonal with respect to the dual Killing form. 
As in \cite{FH}, we take the set of simple roots 
to be $\Delta=\{t_1-t_2,t_2-t_3,\dots, t_{g-1}-t_g, 2t_g\}$.  
The fundamental dominant weights are $\lambda_i=t_1+\cdots +t_i$, 
with $1\le i\le g$. 
The finite-dimensional irreducible representations of $\sp_{2g} (\C)$ 
are parametrized by tuples $(a_1,\dots, a_{g})$  of non-negative integers;    
to such a tuple, there corresponds an irreducible representation 
$V(\lambda)$, with highest weight $\lambda =\sum_{i=1}^{g} a_i \lambda_i$. 

The following theorem recovers a result from \cite{DP}.

\begin{theorem}
\label{thm:rrtor}
For each $g\ge 4$, the resonance variety 
$\RR(\T_g)$ vanishes and $\dim \Wey(V,K) <\infty$. 
\end{theorem}

\begin{proof}
Let $V^*= H^1 (\T_g, \C)$, and let $K^{\perp}\subset V^*\wedge V^*$ 
be the kernel of the cup-product map in degree $1$.   
Work of  Hain \cite{Ha} identifies these 
$\sp_{2g}(\C)$-representation spaces, as follows: 
$V^*=V(\lambda_3)$ and $K^{\perp}=V(2\lambda_2)\oplus V(0)$.  
Moreover, the decomposition of $V^*\wedge V^*$ into 
irreducibles is multiplicity-free.  

Let us apply Corollary \ref{cor:mult free} to this  
situation. 
The only simple root $\beta\in \Delta$ such that 
$(\lambda_3,\beta)\ne 0$ is $\beta=t_3-t_4$. 
Clearly, $2\lambda_3 - \beta=\lambda_2+\lambda_4$ 
does not belong to $\VV(K^{\perp})$.  Hence, 
$\RR(V,K)=\{0\}$. By Lemma \ref{lem:supp}, $\dim \Wey(V,K) <\infty$.
\end{proof}

\begin{remark}
\label{rem:g3}
When $g=3$, the above proof breaks down. 
Indeed, take $\beta=2t_3\in \Delta$; 
then $(\lambda^*,\beta)\ne 0$ and 
$2\lambda^*-\beta=2\lambda_2$ belongs to $\VV(K^{\perp})$. 
Hence, $\RR(V,K)\ne \{0\}$.  In fact, 
$K^{\perp}=V^* \wedge V^*$ in this case, and thus $\RR(V,K)=V^*$.
\end{remark}

\begin{corollary}
\label{cor:btg}
For each $g\ge 6$, let $V=V( \lambda_3)$ and 
$K^{\perp} = V(2\lambda_2)\oplus V(0)$. Then 
$\dim \widetilde{B} (\T_g) = \dim \Wey (V,K)$.
\end{corollary}

\begin{proof}
As  shown in \cite{Ha}, the group $\T_g$ is $1$-formal, 
as long as $g\ge 6$.  Furthermore, as shown in \cite{DHP}, 
the $\widetilde{R}$-module $\BB(\T_g)$ is nilpotent,  
provided $g\ge 4$.   The conclusion follows from 
Corollary \ref{cor:formal}\eqref{f3}, 
and the fact that all representations of $\sp_{2g}(\C)$ 
are self-dual.
\end{proof}

\begin{ack}
This work was started while the first author visited Northeastern 
University, in Spring, 2011; he thanks the Northeastern Mathematics 
Department for its support and hospitality.   Part of this work was 
carried out while the two authors visited the Max Planck Institute 
for Mathematics in Bonn in April--May 2012; we both wish to 
thank MPIM for its support and excellent research atmosphere. 
The work was completed while the second author visited the 
Institute of Mathematics of the Romanian Academy in June, 
2012; he thanks IMAR for its support and hospitality.

We wish to thank Don King for suggesting a simplified proof for  
Lemma \ref{lem:weights}.
\end{ack}

\newcommand{\arxiv}[1]
{\texttt{\href{http://arxiv.org/abs/#1}{arxiv:#1}}}
\newcommand{\arx}[1]
{\texttt{\href{http://arxiv.org/abs/#1}{arXiv:}}
\texttt{\href{http://arxiv.org/abs/#1}{#1}}}
\renewcommand{\MR}[1]
{\href{http://www.ams.org/mathscinet-getitem?mr=#1}{MR#1}}
\newcommand{\doi}[1]
{\texttt{\href{http://dx.doi.org/#1}{doi:#1}}}

\end{document}